\documentclass[11pt,a4paper]{article}
\usepackage[T1]{fontenc}
\usepackage[utf8]{inputenc}
\usepackage[english]{babel}
\usepackage{lmodern,microtype,amsmath,amssymb,amsthm}
\usepackage[margin=29mm,top=25mm,bottom=28mm]{geometry}
\usepackage[hidelinks]{hyperref}
\newtheorem{theorem}{Theorem}[section]
\newtheorem{corollary}[theorem]{Corollary}
\newtheorem{proposition}[theorem]{Proposition}
\theoremstyle{remark}

\newtheorem{example}[theorem]{Example}
\newcommand{\dgh}{d_{\mathrm{GH}}}
\providecommand{\U}{}
\renewcommand{\U}{\mathcal U}
\newcommand{\Ult}{\textbf{U}}
\newcommand{\chainclass}{\mathcal C}
\newcommand{\dist}{\operatorname{dist}_{\mathrm{GH}}}
\DeclareMathOperator{\diam}{diam}
\DeclareMathOperator{\dis}{dis}
\newcommand{\vp}{\varepsilon}
\newcommand{\RR}{\mathbb R}
\hypersetup{pdftitle={Ultrametric approximation and Gromov--Hausdorff clouds},pdfauthor={I. N. Mikhailov}}
\title{Ultrametric approximation and Gromov--Hausdorff clouds}
\author{I. N. Mikhailov}
\date{}
\begin{document}
\maketitle
\begin{abstract}
We compute the Gromov--Hausdorff distance from a metric space to the class of ultrametric spaces in terms of its ultrametric defect and study when a compact ultrametric space attains this distance. For the product of a compact ultrametric space and a compact connected metric space, we give a criterion for the existence of such an approximation. For every integer $k\geqslant1$, we obtain an exact formula for the distance from a bounded metric space to the class of ultrametric spaces with at most $k$ points and prove that the minimum is attained. We characterize the clouds that contain a path-connected space. For compact $X$, we construct a geodesic of length $\diam\Ult(X)/2$ to a nearest connected compact space.

\end{abstract}
\begin{center}\small\emph{Keywords.} Gromov--Hausdorff distance, subdominant ultrametric, ultrametric defect, chain-connectedness, cloud.\end{center}

\section{Introduction}
How well can a given metric space be approximated by an ultrametric space? The answer depends on the conditions imposed on the approximating space. For a finite metric space $(X,d_X)$, one can choose an ultrametric $v$ on the same set of points so that the Gromov--Hausdorff distance between $(X,d_X)$ and $(X,v)$ equals the infimum of the distances from $(X,d_X)$ to all ultrametric spaces.

For an infinite space, two further questions arise. Can the approximating ultrametric space be chosen compact? Does a best compact approximation exist?

The construction underlying these questions has a simple interpretation. Join two points $x,x'$ of a metric space $X$ by a finite chain $x_0=x,x_1,\ldots,x_m=x'$ and take the length of its longest step. The infimum of these lengths over all chains gives the chain distance. This replaces the original metric by a pseudoultrametric. Identifying points at zero chain distance gives an ultrametric space $\Ult(X)$, called the \emph{ultrametrization} of $X$. The idea goes back to Cantor's description of connectedness by chains whose longest step is arbitrarily small \cite[\S 10, pp.~575--576]{Cantor}. The resulting subdominant ultrametric was studied in \cite{BMM,Lemin}. Its stability with respect to the Gromov--Hausdorff distance is closely related to the stability of single-linkage clustering \cite{CM,LMS}.

For finite spaces, Chepoi and Fichet solved the problem of approximating a metric by an ultrametric on the same point set in the uniform norm \cite[\S 3, p.~607]{CF}. Their approximation adds a constant to the subdominant ultrametric off the diagonal. We show that allowing arbitrary correspondences does not improve the optimal error and that the same argument applies to arbitrary metric spaces. More precisely, let $u_X$ denote the chain pseudoultrametric, and define the \emph{ultrametric defect} of $X$ by
\[
 E(X)=\sup_{x,x'\in X}\bigl(d_X(x,x')-u_X(x,x')\bigr),
\]
Then the distance from $X$ to the class of all ultrametric spaces is one quarter of its ultrametric defect, namely $E(X)/4$ (Theorem~\ref{thm:fit}). If $E(X)<\infty$, the space $(X,v)$ attains the infimum, where $v$ is given by~\eqref{eq:shift}. Thus the approximation error measures how much the chain construction can shorten a distance. For finite $X$, Proposition~52 of \cite{MSW} identifies $E(X)/2$ with the least uniform error of an ultrametric on the given point set, called the absolute ultrametricity of $X$. Theorem~\ref{thm:fit} shows that even a change of point set cannot lower the distortion below this fixed-label optimum.

If $X$ is totally bounded, the infima over finite and over compact ultrametric spaces are also equal to $E(X)/4$. Even when $X$ is compact, this equality alone does not guarantee that a compact ultrametric space attains the infimum. Theorem~\ref{thm:cylinder} settles the question for a family of products. If $B$ is compact and ultrametric and $A$ is compact and connected with diameter $L>0$, then the distance from the product $B\times_1 A$, equipped with the $\ell_1$ metric, to the class of compact ultrametric spaces is $L/4$. A compact ultrametric space $V$ with $\dgh(B\times_1 A,V)=L/4$ exists if and only if $B$ is finite.

For every integer $k\geqslant1$, we also consider approximations of a bounded metric space $X$ by ultrametric spaces with at most $k$ points. Theorem~\ref{thm:part} expresses the distance $e_k(X)$ to this class in terms of partitions of $X$ and proves that the minimum is attained. For compact $X$ and unrestricted metric targets, only the largest block diameter enters the partition formula of \cite{MSS}. Our formula has a second term: it measures the discrepancy between the greatest distance across two blocks and the smallest bottleneck along a path of blocks joining them. For a bounded chain-connected space $X$ (that is, $\Ult(X)$ is a point) of diameter $D$, the formula simplifies to
\[
 e_k(X)=\max\{D/4,H_k(X)/2\},
\]
where $H_k(X)$ is the infimum, over partitions into at most $k$ nonempty sets, of the largest diameter of a block. If $D>0$, the minimum is attained by a simplex on $m\leqslant k$ points with distance $D/2$ between distinct points. The case $m=1$ is allowed (Corollary~\ref{cor:chaink}).

The cloud $[X]$ consists of all spaces at finite Gromov--Hausdorff distance from $X$. This cloud contains a path-connected space precisely when $\Ult(X)$ has finite diameter. Here we determine the exact distance to the class of path-connected spaces:
\[
 \inf_{Y\ \text{path-connected}}\dgh(X,Y)=\tfrac12\diam\Ult(X).
\]
For compact $X$, we construct a geodesic in the Gromov--Hausdorff class from $X$ to a nearest connected compact space (Theorem~\ref{thm:shorten}). Together with the ultrametric approximation formula, this shows that the cloud of bounded spaces is the only cloud containing both an ultrametric space and a chain-connected space (Corollaries~\ref{cor:clouds} and~\ref{cor:ultraclouds}).

\section{Definitions and preliminary properties}\label{sec:definitions}
\subsection{Gromov--Hausdorff distance and clouds}
All spaces are nonempty, and distances between their points are finite. We use the ordinary, unpointed Gromov--Hausdorff distance, allowing the value $+\infty$. A \emph{correspondence} is a relation $R\subset X\times Y$ whose projections onto both factors are surjective. Its \emph{distortion} is the supremum of $|d_X(x,x')-d_Y(y,y')|$ over $(x,y),(x',y')\in R$. We use the formula \cite[Theorem~7.3.25]{BBI}
\begin{equation}
 2\dgh(X,Y)=\inf_R\dis R.\label{eq:GH}
\end{equation}
For a class $\mathcal A$, put $\dist(X,\mathcal A)=\inf_{Y\in\mathcal A}\dgh(X,Y)$.

Following \cite[\S 1.2]{Clouds}, we call the class
\[
 [X]=\{Y:\dgh(X,Y)<\infty\}.
\]
the \emph{cloud} of $X$. When discussing the metric on a cloud or maps between clouds, we identify spaces at Gromov--Hausdorff distance zero. Saying that $[X]$ contains a path-connected space means that it has a path-connected representative. For the one-point space $\Delta_1$, formula~\eqref{eq:GH} gives $\dgh(X,\Delta_1)=\diam X/2$. Thus $[\Delta_1]$ is precisely the cloud of bounded spaces. We shall determine which clouds contain a path-connected space and which contain an ultrametric space.

A \emph{simplex} is a metric space in which the distance between any two distinct points has the same value, called the side length. We also regard a one-point space as a simplex, in which case the condition on pairwise distances is vacuous. A simplex may be finite or infinite.

\subsection{The chain pseudoultrametric and ultrametrization}

A \emph{chain} joining points $a,b$ of a metric space $X$ is a finite sequence $x_0=a,x_1,\ldots,x_m=b$ in $X$, with $m\geqslant1$. Its \emph{mesh} is the length of its longest step, $\max_{1\leqslant j\leqslant m}d_X(x_{j-1},x_j)$. The infimum of the meshes of all chains with prescribed endpoints defines the chain pseudoultrametric. We denote it by $u_X$ and its metric quotient by $\Ult(X)$:
\begin{equation}
 \begin{split}
 u_X(x,x')&=\inf_{x=x_0,\ldots,x_m=x'}\max_{1\leqslant j\leqslant m}d_X(x_{j-1},x_j),\\
 \Ult(X)&=(X,u_X)/\{u_X=0\}.
 \end{split}\label{eq:chain}
\end{equation}
The infimum runs over finite chains with $m\geqslant1$. We call $\Ult(X)$ the \emph{ultrametrization} of $X$.

Concatenation of chains gives the ultrametric triangle inequality. The one-step chain shows that $u_X\leqslant d_X$. If a pseudoultrametric is dominated by $d_X$, its value at the endpoints of any chain is at most the mesh of the chain, by the ultrametric triangle inequality. It is therefore dominated by $u_X$. Thus $u_X$ is the greatest such pseudoultrametric. This construction was studied in \cite{BMM}.

Lemin \cite[Part~I]{Lemin} established the following universal property. Let $q_X\colon X\to\Ult(X)$ be the quotient map. For every ultrametric space $V$ and every $1$-Lipschitz map $f\colon X\to V$, there exists a unique $1$-Lipschitz map $\bar f\colon\Ult(X)\to V$ such that $f=\bar f\circ q_X$. Indeed, applying the ultrametric triangle inequality along each chain and taking the infimum gives
\[
 d_V(f(x),f(x'))\leqslant u_X(x,x').
\]
Hence $f$ is constant on zero-chain-distance classes and induces a $1$-Lipschitz map on the quotient. Its uniqueness follows from the surjectivity of $q_X$. In particular, for ultrametric $V$, the space $\Ult(V)$ is canonically isometric to $V$.

A space is \emph{chain-connected} if every two points can be joined by chains of arbitrarily small mesh \cite[Introduction]{Plaut}, or equivalently if $\Ult(X)$ is a point.

Every connected metric space is chain-connected. Indeed, fix $\vp>0$ and declare two points equivalent if they can be joined by a finite chain of mesh less than $\vp$. The one-step chain with coincident endpoints gives reflexivity. Reversing and concatenating chains give symmetry and transitivity. Every class is open, since a chain ending at $x$ can be extended by one step to any point in the open ball of radius $\vp$ about $x$. Its complement, a union of other classes, is open as well. If there were two or more classes, any one of them and its complement would separate the space into two nonempty open sets. Thus there is only one class. Since $\vp>0$ was arbitrary, every two points can be joined by chains of arbitrarily small mesh.

Here \emph{connectedness} has its usual topological meaning: the space cannot be partitioned into two nonempty disjoint open sets. A space is \emph{path-connected} if any two points can be joined by a continuous path $[0,1]\to X$. Path-connectedness implies connectedness, but the converse can fail. For path-connected spaces see also \cite[Proposition~C.1]{LMS}.

\subsection{Stability and products}
For metric spaces $X,Y$ and $1\leqslant p\leqslant\infty$, write $X\times_p Y$ for their \emph{$\ell_p$ product}, the set $X\times Y$ equipped with the metric
\[
 d_{X\times_p Y}\bigl((x,y),(x',y')\bigr)=
 \begin{cases}
 \bigl(d_X(x,x')^p+d_Y(y,y')^p\bigr)^{1/p},&1\leqslant p<\infty,\\
 \max\{d_X(x,x'),d_Y(y,y')\},&p=\infty.
 \end{cases}
\]
In particular, $d_{X\times_1 Y}((x,y),(x',y'))=d_X(x,x')+d_Y(y,y')$. This metric restricts to the original metrics on coordinate slices and dominates each coordinate distance.

We recall the Gromov--Hausdorff stability of ultrametrization. The finite-space result appeared in \cite[Proposition~26]{CM}, and the bounded-space version is stated in \cite[Appendix~C, Theorem~H]{LMS}. For completeness, we include the proof in a form that also covers unbounded spaces (Proposition~\ref{rem:basic}(i)). We make no claim of originality for this stability result.

\begin{proposition}[Stability and products]\label{rem:basic}
\textup{(i)} If a correspondence $R\subset X\times Y$ has distortion $c<\infty$, then for $(x,y),(x',y')\in R$,
\begin{equation}
 |u_X(x,x')-u_Y(y,y')|\leqslant c.\label{eq:lift}
\end{equation}
Consequently,
\[
 \dgh\bigl(\Ult(X),\Ult(Y)\bigr)\leqslant\dgh(X,Y),
\]
and $\Ult$ induces a $1$-Lipschitz map $[X]\to[\Ult(X)]$.

\textup{(ii)} Suppose a metric $\rho$ on $X\times Y$ restricts to the original metrics on all coordinate slices $\{x\}\times Y$ and $X\times\{y\}$. Suppose also that for every pair of points of the product, \[\rho((x,y),(x',y'))\geqslant\max\{d_X(x,x'),d_Y(y,y')\}.\] Then
\begin{equation}
 u_{X\times_\rho Y}\bigl((x,y),(x',y')\bigr)=\max\bigl\{u_X(x,x'),u_Y(y,y')\bigr\}.\label{eq:product}
\end{equation}
In particular, $\Ult(X\times_\rho Y)$ is canonically isometric to $\Ult(X)\times_\infty\Ult(Y)$.

\textup{(iii)} For chain-connected $A$, ultrametric spaces $V,W$, and $1\leqslant p\leqslant\infty$,
\[
 \dgh(V\times_p A,W\times_p A)=\dgh(V,W).
\]
For every $A$, the map $F_A(X)=X\times_p A$ is $1$-Lipschitz from $[\Delta_1]$ to $[A]$. If $A$ is chain-connected, the restriction of $F_A$ to bounded ultrametric spaces, considered up to zero GH distance, is an isometric embedding into $[A]$. For each such space $V$, the space $\Ult(F_A(V))$ is canonically isometric to $V$.
\end{proposition}

\begin{proof}
\textup{(i)} Fix $(x,y),(x',y')\in R$ and $\vp>0$. Choose a chain $x=x_0,\ldots,x_m=x'$ of mesh less than $u_X(x,x')+\vp$. Since $R$ projects onto $X$, we can choose $y_i$ with $(x_i,y_i)\in R$, taking $y_0=y$ and $y_m=y'$. Then
\[
 \begin{split}
 u_Y(y,y')&\leqslant\max_i d_Y(y_{i-1},y_i) \leqslant\\
 &\leqslant\max_i d_X(x_{i-1},x_i)+c<u_X(x,x')+\vp+c.
 \end{split}
\]
Letting $\vp\downarrow0$ and repeating the argument with $X,Y$ interchanged proves~\eqref{eq:lift}. Let $q_X,q_Y$ be the quotient maps. The relation
\[
 \overline R=\{(q_X(x),q_Y(y)):(x,y)\in R\}
\]
is a correspondence between $\Ult(X)$ and $\Ult(Y)$. By~\eqref{eq:lift}, its distortion is at most $\dis R$. Taking the infimum in~\eqref{eq:GH} proves the desired estimate when $\dgh(X,Y)<\infty$. The infinite case is immediate. The estimate also shows that $\Ult$ preserves zero GH distance and maps $[X]$ into $[\Ult(X)]$.

\textup{(ii)} Project a chain in $X\times Y$ onto each factor. By hypothesis, every coordinate step is at most the corresponding $\rho$-step. Thus the mesh of the original chain is at least both $u_X(x,x')$ and $u_Y(y,y')$. Taking the infimum gives the lower bound in~\eqref{eq:product}. Conversely, for $\vp>0$ choose coordinate chains of meshes less than $u_X(x,x')+\vp$ and $u_Y(y,y')+\vp$. Follow the first in $X\times\{y\}$ and the second in $\{x'\}\times Y$. By the slice assumption, the resulting chain has mesh less than
\[
 \max\{u_X(x,x'),u_Y(y,y')\}+\vp.
\]
Letting $\vp$ tend to zero gives equality. The surjection $(x,y)\mapsto(q_X(x),q_Y(y))$ preserves chain distances when the target product carries the $\ell_\infty$ metric. It therefore induces the canonical isometry of metric quotients
\[
 \Ult(X\times_\rho Y)\cong\Ult(X)\times_\infty\Ult(Y).
\]

\textup{(iii)} For arbitrary spaces $X,Y,A$ and a correspondence $R\subset X\times Y$ of finite distortion, put
\[
 R_A=\{((x,a),(y,a)):(x,y)\in R,\ a\in A\}.
\]
This is a correspondence between $X\times_p A$ and $Y\times_p A$. For two of its pairs, put $s=d_X(x,x')$, $t=d_Y(y,y')$, and $h=d_A(a,a')$. The reverse triangle inequality in $\ell_p^2$ gives
\[
 \bigl|\|(s,h)\|_p-\|(t,h)\|_p\bigr|\leqslant\|(s-t,0)\|_p
 =|s-t|\leqslant\dis R.
\]
Thus $\dis R_A\leqslant\dis R$. Taking infima gives
$\dgh(X\times_p A,Y\times_p A)\leqslant\dgh(X,Y)$, including the immediate case of an infinite right-hand side. If $A$ is chain-connected, part~\textup{(ii)} identifies $\Ult(V\times_p A)$ with $V$ and $\Ult(W\times_p A)$ with $W$. Part~\textup{(i)} now gives the reverse inequality
\[
 \dgh(V,W)\leqslant\dgh(V\times_p A,W\times_p A),
\]
even when the left-hand side is infinite. Finally, $F_A(\Delta_1)$ is isometric to $A$. The upper bound therefore gives $F_A([\Delta_1])\subset[A]$ and shows that $F_A$ is well defined on zero-distance classes. For chain-connected $A$, part~\textup{(i)} gives $\Ult([A])\subset[\Delta_1]$, and part~\textup{(ii)} yields $\Ult(F_A(V))\cong V$ for each bounded ultrametric space $V$. This proves the cloud assertions.
\end{proof}

\section{Best ultrametric approximation}
Let $\U$ denote the class of all ultrametric spaces, $\U_{\rm c}$ the compact ones, and $\U_{\rm f}$ the finite ones. We define the \emph{ultrametric defect} of a metric space $X$ by
\[
 E(X)=\sup_{x,x'\in X}\bigl(d_X(x,x')-u_X(x,x')\bigr).
\]
The ultrametric defect can be infinite even though each term in the supremum is finite.

\begin{theorem}\label{thm:fit}
For every metric space $X$,
\begin{equation}
 \dist(X,\U)=\tfrac14 E(X).\label{eq:fit}
\end{equation}
If $E(X)<\infty$, the infimum in~\eqref{eq:fit} is attained by $(X,v)$, where the ultrametric $v$ on the point set of $X$ is given by
\begin{equation}
 v(x,x')=u_X(x,x')+\tfrac12 E(X)\,\mathbf 1_{\{x\ne x'\}}.\label{eq:shift}
\end{equation}
In other words, $\dgh((X,d_X),(X,v))=\dist(X,\U)=E(X)/4$.
If $X$ is totally bounded, then $\dist(X,\U_{\rm f})=\dist(X,\U_{\rm c})=E(X)/4$.
\end{theorem}

\begin{proof}
Let $R$ be a correspondence of finite distortion $c=\dis R$ between $X$ and an ultrametric space $V$. For $x,x'\in X$, choose $v,v'\in V$ with $(x,v),(x',v')\in R$. By~\eqref{eq:lift},
\[
 d_X(x,x')\leqslant d_V(v,v')+c=u_V(v,v')+c\leqslant u_X(x,x')+2c.
\]
Taking the supremum over $x,x'$ gives $E(X)\leqslant2c$. Taking the infimum over correspondences yields $E(X)\leqslant4\dgh(X,V)$, and taking the infimum over all ultrametric spaces $V$ gives the required lower bound. If $E(X)=\infty$, no correspondence of finite distortion with any ultrametric space exists, so both sides of~\eqref{eq:fit} are infinite. If $0<E(X)<\infty$, adding $E(X)/2$ off the diagonal preserves the ultrametric triangle inequality and separates points. If $E(X)=0$, then $u_X=d_X$. For $x\ne x'$ we have $0\leqslant d_X(x,x')-u_X(x,x')\leqslant E(X)$, so
\[
 \begin{split}
 |d_X(x,x')-v(x,x')|
 =\left|d_X(x,x')-u_X(x,x')-\frac{E(X)}2\right| \leqslant\frac{E(X)}2.
 \end{split}
\]
On the diagonal both metrics vanish. Thus the diagonal correspondence between $(X,d_X)$ and $(X,v)$ has distortion at most $E(X)/2$, whence
\[
 \dgh\bigl((X,d_X),(X,v)\bigr)\leqslant E(X)/4.
\]
The lower bound proves equality and shows that the infimum is attained. For finite $X$ this is precisely the shift in \cite[\S 3]{CF}.

For the last assertion, fix $\vp>0$ and choose a finite $\vp$-net $F\subset X$. Then $\dgh(F,X)\leqslant\vp$. Across a correspondence between $F$ and $X$ of distortion $c$, the original distances differ by at most $c$ and, by~\eqref{eq:lift}, so do the chain distances. Thus $E(F)\leqslant E(X)+2c$. Taking the infimum over correspondences yields $E(F)\leqslant E(X)+4\dgh(F,X)$. Apply~\eqref{eq:shift} to $F$ with its induced metric, computing $u_F$ along chains in $F$. The result just proved gives a finite ultrametric $V_\vp$ with $\dgh(F,V_\vp)=E(F)/4$. Hence
\[
 \dgh(X,V_\vp)\leqslant\dgh(X,F)+\dgh(F,V_\vp) =\dgh(X,F)+E(F)/4\leqslant E(X)/4+2\vp.
\]
The inclusions $\U_{\rm f}\subset\U_{\rm c}\subset\U$ give
\[
 E(X)/4\leqslant\dist(X,\U_{\rm c})\leqslant\dist(X,\U_{\rm f})\leqslant E(X)/4+2\vp.
\]
A totally bounded space is bounded, so $E(X)<\infty$. As these inequalities hold for every $\vp>0$, both infima equal $E(X)/4$.
\end{proof}

\paragraph{An example of infinite ultrametric defect.}
Take $X=\mathbb Z$ with $d(m,n)=|m-n|$. Any chain between distinct integers has a step of length at least one, whereas consecutive integers give a chain of mesh one. Hence $u_X(m,n)=1$ for $m\ne n$, and
\[
 E(X)=\sup_{m\ne n}(|m-n|-1)=\infty.
\]
By Theorem~\ref{thm:fit}, $\mathbb Z$ is at infinite GH distance from every ultrametric space.

It follows from Theorem~\ref{thm:fit} that the ultrametric members of each cloud form a closed subclass in the Gromov--Hausdorff metric. Indeed, if ultrametric spaces $V_n$ converge to $X$ in Gromov--Hausdorff distance, then
\[
 0\leqslant E(X)/4=\dist(X,\U)\leqslant\dgh(X,V_n)\longrightarrow0.
\]
Consequently $E(X)=0$, so $d_X=u_X$ and $X$ is ultrametric. For a bounded chain-connected space $X$ of diameter $D>0$, formula~\eqref{eq:shift} makes $X$ a simplex of side $D/2$, and the approximation error is $D/4$. The one-point quotient $\Ult(X)$ gives error $D/2$.

\begin{theorem}[Compact attainment for products]\label{thm:cylinder}
Let $B$ be a compact ultrametric space and $A$ a compact connected metric space of diameter $L>0$. Put $X=B\times_1 A$. Then
\[
 \dist(X,\U_{\rm c})=L/4.
\]
A compact ultrametric space $V$ satisfying $\dgh(X,V)=L/4$ exists if and only if $B$ is finite. In particular, the closed set of compact ultrametric spaces in the Gromov--Hausdorff space of compact metric spaces can fail to contain a nearest point to a given compact space.
\end{theorem}

\begin{proof}
Connectedness of $A$ implies chain-connectedness, so $u_A=0$. By~\eqref{eq:product}, the chain distance between $(b,s)$ and $(b',t)$ equals $d_B(b,b')$. The product distance exceeds the chain distance by $d_A(s,t)$, whence $E(X)=L$. Theorem~\ref{thm:fit} gives the infimum.

Put $h=L/2$. If $B$ is finite, use a finite cover of $A$ by balls of radius $h/2$ to obtain a finite partition $A_1,\ldots,A_m$ into sets of diameter at most $h$. Take $m$ copies of each $b$, indexed by the blocks of the partition. Give distinct copies of the same point distance $h$, and copies of distinct $b,b'$ distance $h+d_B(b,b')$. The ultrametric triangle inequality on $B$ makes this an ultrametric. Send $(b,s)$ to the copy of $b$ indexed by the block containing $s$. Points with the same image are at distance at most $h$. For distinct images the absolute error is $|d_A(s,t)-h|\leqslant h$. The graph of the map has distortion at most $h$.

We now show that no nearest compact ultrametric space exists when $B$ is infinite. Write $N_a(V)$ for the number of equivalence classes under $v\sim v'$ if and only if $d_V(v,v')\leqslant a$ in an ultrametric space $V$. If a correspondence from $X$ to $V$ has finite distortion $c$, then $c\geqslant h$ and
\begin{equation}
 N_c(V)\geqslant N_{2c-L}(B).\label{eq:entropy}
\end{equation}
Let $R$ be this correspondence. For each $b\in B$, put
\[
 R_b=\{v\in V:\text{ there exists }s\in A\text{ with }((b,s),v)\in R\}.
\]
By~\eqref{eq:lift}, any two points of $R_b$ are at distance at most $c$, so $R_b$ lies in one $c$-class. If $R_b$ and $R_{b'}$ lie in the same $c$-class, choose $s,t\in A$ with $d_A(s,t)=L$ and compare images of $(b,s)$ and $(b',t)$. Their distance is at most $c$, so $L+d_B(b,b')\leqslant2c$. Thus distinct $(2c-L)$-classes in $B$ require distinct $c$-classes in $V$.

If a compact $V$ attained $L/4$, compactness of $X$ and $V$ would give an optimal correspondence of distortion $h$ \cite[Remark~2.1]{MSS}. At $c=h$, \eqref{eq:entropy} gives $N_h(V)\geqslant N_0(B)=|B|$. A finite cover of $V$ by balls of radius $h/2$ meets all $h$-classes, and each ball lies in a single class. Hence $N_h(V)$ is finite, so $B$ is finite.
\end{proof}

If $A=[0,L]$ and $B$ is finite, the least cardinality of a compact ultrametric space $V$ satisfying $\dgh(B\times_1[0,L],V)=L/4$ is $2|B|$. The partition $[0,h),[h,L]$ gives the upper bound. Conversely, take an optimal correspondence of distortion $h=L/2$. The two endpoints of each fiber, being at distance $L>h$, require distinct images. The proof of~\eqref{eq:entropy} puts different fibers in disjoint $h$-classes since $2h-L=0$. Thus $V$ has at least $2|B|$ points.

For an explicit infinite base, take $B=\{0\}\cup\{1/n:n\geqslant1\}$ with $d_B(s,t)=\max\{s,t\}$ for $s\ne t$. This is a compact ultrametric space since $1/n\to0$.

\section{Approximation by spaces with at most \texorpdfstring{$k$}{k} points}
For bounded $X$, put $e_k(X)=\dist(X,\U_{\leqslant k})$, where $\U_{\leqslant k}$ consists of ultrametric spaces with at most $k$ points. For unrestricted metric approximations, \cite[Theorem~1.7]{MSS} gives the following equality for every compact metric space $X$ and every integer $k\geqslant1$ for which $X$ admits a partition into $k$ nonempty sets:
\[
 \inf_{\substack{Y\ \text{metric}\\ |Y|=k}}\dgh(X,Y)
 =\frac12\inf_{\substack{\mathcal P=\{P_1,\ldots,P_k\}\\ P_i\ne\varnothing}}
       \max_{1\leqslant i\leqslant k}\diam P_i.
\]
The infimum on the right runs over all partitions into $k$ nonempty sets. The blocks need not be closed. In the notation of the cited paper, this is the identity $\mathbf{Sketch}_k(X)=\tfrac12\mathbf{Shatter}_k(X)$. For spaces with at most $k$ points it becomes
\[
 \inf_{\substack{Y\ \text{metric}\\1\leqslant |Y|\leqslant k}}\dgh(X,Y)
 =\frac12\inf_{|\mathcal P|\leqslant k}\max_{P\in\mathcal P}\diam P.
\]
This follows by taking the minimum over all admissible numbers of blocks from $1$ to $k$. The spaces $Y$ on the left may carry any metric. For the ultrametric approximation problem defining $e_k(X)$ one must also control distances between blocks.

Given a partition $\mathcal P=\{P_1,\ldots,P_m\}$ into nonempty sets, put $H(\mathcal P)=\max_i\diam P_i$. For $i\ne j$, let $a_{ij}$ and $b_{ij}$ denote the infimum and supremum, respectively, of distances from a point of $P_i$ to a point of $P_j$, and set $a_{ii}=0$. In the complete graph on $1,\ldots,m$, assign weight $a_{ij}$ to the edge $\{i,j\}$. For $i\ne j$, let $w_{ij}$ be the minimum, over all paths from $i$ to $j$, of the greatest edge weight on the path. Put $w_{ii}=0$. This gives the largest pseudoultrametric dominated by $a$, even when $a$ does not satisfy the triangle inequality.

\begin{theorem}\label{thm:part}
For bounded $X$ and every integer $k\geqslant1$,
\begin{equation}
 2e_k(X)=\inf_{|\mathcal P|\leqslant k}
 \max\left\{H(\mathcal P),\ \tfrac12\max_{i\ne j}(b_{ij}-w_{ij})\right\},\label{eq:k}
\end{equation}
where the maximum over an empty set is zero. There exists an ultrametric space $V$ with at most $k$ points such that $\dgh(X,V)=e_k(X)$.
\end{theorem}

\begin{proof}
Let $R$ be a correspondence of distortion $c$ between $X$ and an ultrametric space with at most $k$ points. Choose one image for each point of $X$ and discard the unused target points. The fibers of the resulting map form a partition with $H(\mathcal P)\leqslant c$. The distances $v_{ij}$ between distinct images satisfy
\[
 b_{ij}-c\leqslant v_{ij}\leqslant a_{ij}+c.
\]
The ultrametric triangle inequality along a path in the graph improves the upper bound to $w_{ij}+c$, whence $b_{ij}-w_{ij}\leqslant2c$.

Conversely, fix a partition with $m\leqslant k$ blocks, and let $c$ be the maximum in~\eqref{eq:k}. Put $v_{ii}=0$ and $v_{ij}=w_{ij}+c$ for $i\ne j$. Adding the same nonnegative constant off the diagonal preserves the ultrametric triangle inequality: for distinct indices both sides increase by $c$, and for repeated indices the inequality is immediate. Since $w_{ij}\leqslant a_{ij}$ and $b_{ij}-w_{ij}\leqslant2c$, we have
\[
 b_{ij}-c\leqslant v_{ij}\leqslant a_{ij}+c\qquad(i\ne j).
\]
Map each $P_i$ to $[i]$ in the metric quotient of $v$. Within a block the error is at most $H(\mathcal P)\leqslant c$. Between distinct blocks the original distance lies in $[a_{ij},b_{ij}]$, so the displayed bounds give error at most $c$, even if $[i]=[j]$. The map is onto a space of at most $k$ points, and its graph is the required correspondence.

Finally, take a minimizing sequence $V_j$ of approximating spaces. Their diameters are uniformly bounded because $\diam V_j\leqslant\diam X+2\dgh(X,V_j)$. Represent each $V_j$ by a pseudoultrametric $v^{(j)}$ on $\{1,\ldots,k\}$, repeating point labels if necessary. A subsequence of these finite distance matrices converges entrywise to a pseudoultrametric $v$. Write $V$ for its metric quotient. Along this subsequence, matching equal labels gives
\[
 \dgh(V_j,V)\leqslant\frac12\max_{a,b}|v^{(j)}_{ab}-v_{ab}|\longrightarrow0.
\]
By the triangle inequality,
\[
 |\dgh(X,V_j)-\dgh(X,V)|\leqslant\dgh(V_j,V)\longrightarrow0.
\]
Thus $\dgh(X,V)=e_k(X)$.
\end{proof}

\begin{corollary}\label{cor:chaink}
Let $X$ be a bounded chain-connected space of diameter $D$, and put $H_k(X)=\inf_{|\mathcal P|\leqslant k}H(\mathcal P)$. Then
\begin{equation}
 e_k(X)=\max\{D/4,H_k(X)/2\}.\label{eq:chaink}
\end{equation}
If $D>0$, there is an $m\in\{1,\ldots,k\}$ such that an $m$-point simplex with distance $D/2$ between distinct points attains $e_k(X)$. For $m=1$ the distance condition is vacuous.
\end{corollary}

\begin{proof}
Theorem~\ref{thm:fit} gives the first lower bound. Selecting one image for each point of $X$ in a correspondence yields a partition into fibers, and the definition of $H_k(X)$ gives the second lower bound. For the upper bound, assign the blocks of a partition to distinct points of a simplex of side $D/2$. Between distinct blocks, every distance in $[0,D]$ incurs an error of at most $D/2$. Within a block, the error is at most $H(\mathcal P)$. Taking the infimum gives the formula. Let $S_m$ be the $m$-point simplex of side $D/2$, for $1\leqslant m\leqslant k$, with $S_1$ the one-point space. For fixed $m$ this simplex is unique up to isometry. The correspondence constructed for each partition gives
\[
 \min_{1\leqslant m\leqslant k}\dgh(X,S_m)
 \leqslant\inf_{|\mathcal P|\leqslant k}
 \frac12\max\{D/2,H(\mathcal P)\}=e_k(X).
\]
The reverse inequality follows from the definition of $e_k(X)$, since each $S_m$ is an admissible ultrametric approximation. The minimum over the finite family $S_1,\ldots,S_k$ is attained, so one of these simplices realizes $e_k(X)$.
\end{proof}

A partition formula for the distance to a prescribed simplex was obtained in \cite[Corollary~2.5 of the preprint]{GIT}. Corollary~\ref{cor:chaink} shows that for a chain-connected source, simplices suffice when optimizing over all ultrametrics with a prescribed bound on cardinality.

\begin{example}[The circle with its intrinsic metric]
For the circle of circumference $2\pi$ with its intrinsic metric,
\[
 e_1(S^1)=e_2(S^1)=\frac\pi2,\qquad
 e_3(S^1)=\frac\pi3,\qquad e_k(S^1)=\frac\pi4\quad(k\geqslant4).
\]
Indeed, $H_1=H_2=\pi$ and $H_k=2\pi/k$ for $k\geqslant3$ \cite[Proposition~1.6 and Theorem~1.8]{MSS}. For the three-block lower bound, observe that any set of diameter $a<2\pi/3$ is contained in an arc of length $a$. Take one point as the origin and lift the others to $[-a,a]$. Two lifted coordinates differ by at most $2a<2\pi-a$. Thus an absolute difference greater than $a$ would give a circular distance greater than $a$. Three arcs of length less than $2\pi/3$ cannot cover the circle. Partitions into equal arcs give the upper bounds.
\end{example}

For comparison, let $P_n$ be the vertex set of an inscribed regular $n$-gon with the metric induced by the intrinsic metric of the circle. In \cite[Proposition~C.2]{LMS}, the equality $\dgh(S^1,P_n)=\pi/n$ is obtained from a chain lower bound and an ambient Hausdorff upper bound. Proposition~C.5 of the same paper gives $\dgh(P_m,P_{m+1})=\pi/(m+1)$ for $m\geqslant2$. In \cite[Theorem~3.1]{Tal}, the equality $\dgh(P_n,P_m)=\pi/n-\pi/m$ is proved for $n\mid m$ and $2\leqslant n\leqslant m$. For ultrametric approximations the error stays at $\pi/4$ once four points are allowed.

\section{Distance to chain-connected spaces}
Put $b(X)=\diam\Ult(X)$ and $\chainclass_t=\{Y:b(Y)\leqslant t\}$ for $0\leqslant t<\infty$. Thus $\chainclass_0$ is the class of chain-connected spaces. By~\eqref{eq:lift},
\begin{equation}
 b(X)\leqslant b(Y)+2\dgh(X,Y).\label{eq:b}
\end{equation}
We construct spaces attaining the resulting lower bound by decreasing all chain distances by the same amount and replacing negative values by zero.

Embed $X$ isometrically in a normed space by the Kuratowski embedding and identify it with its image. Explicitly, fix $o\in X$ and send $x$ to $d_X(x,\cdot)-d_X(o,\cdot)$ in $\ell_\infty(X)$. For $r\geqslant0$, put
\[
 Z_r=\{(1-\lambda)x+\lambda y:x,y\in X,\quad
 0\leqslant\lambda\leqslant1,\quad\lambda\|y-x\|\leqslant r\}
\]
with the ambient metric. Thus $Z_0=X$, and $Z_r$ contains the initial portion of length $\min\{r,\|y-x\|\}$ of every segment directed from $x\in X$ to $y\in X$. Write $a_+=\max\{a,0\}$.

\begin{theorem}\label{thm:shorten}
For $x,x'\in X$ and $r\geqslant0$,
\begin{equation}
 u_{Z_r}(x,x')=(u_X(x,x')-2r)_+,\qquad
 b(Z_r)=(b(X)-2r)_+.\label{eq:shorten}
\end{equation}
Consequently,
\begin{equation}
 \dist(X,\chainclass_t)=\tfrac12(b(X)-t)_+,\qquad
 \inf_{Y\ \text{path-connected}}\dgh(X,Y)=\tfrac12 b(X).\label{eq:conn}
\end{equation}
If $b=b(X)<\infty$, then $\dgh(Z_r,Z_s)=s-r$ for $0\leqslant r\leqslant s\leqslant b/2$. The endpoint $Z_{b/2}$ is a nearest chain-connected space. For compact $X$, this gives a geodesic in the space of compact metric spaces from $X$ to a nearest connected compact space.
\end{theorem}

\begin{proof}
On a chord of length $d$, the two endpoint portions in $Z_r$ leave a gap of length $(d-2r)_+$. Subdivide these portions into arbitrarily small steps. Applying this construction to each edge of a chain in $X$ gives the upper bound for $u_{Z_r}$. Conversely, for each vertex of a chain in $Z_r$ choose an endpoint of its chord in $X$ at distance at most $r$, keeping the endpoints $x,x'\in X$ fixed. If the original chain has mesh $\delta$, the new chain in $X$ has mesh at most $\delta+2r$. Taking the infimum over chains in $Z_r$ yields
\[
 u_X(x,x')\leqslant u_{Z_r}(x,x')+2r.
\]
Together with the nonnegativity of chain distance, this gives
\[
 u_{Z_r}(x,x')\geqslant(u_X(x,x')-2r)_+.
\] Every point of $Z_r$ is joined to a chord endpoint by a segment in $Z_r$, so every zero-chain-distance class meets $X$. Taking the supremum gives the second equality in~\eqref{eq:shorten}.

For $0\leqslant r\leqslant s$, truncating each initial chord portion at length $r$ gives $d_H(Z_r,Z_s)\leqslant s-r$ in the ambient space. If $s\leqslant b/2$ and $b$ is finite, \eqref{eq:b} and~\eqref{eq:shorten} give the opposite GH bound. For $b<\infty$, put $r=(b-t)_+/2$. Then $Z_r\in\chainclass_t$ and
\[
 \dgh(X,Z_r)=r.
\]
Thus $Z_r$ attains the infimum of distances to $\chainclass_t$. If $b=\infty$, \eqref{eq:b} gives infinite distance to every $\chainclass_t$. The distance to any path-connected space $Y$ is also infinite because $b(Y)=0$.

For finite $b$ and $r>b/2$, any two points of $X$ can be joined by a chain whose edges are all shorter than $2r$. The corresponding chords lie entirely in $Z_r$, so any two points of $X$ can be joined by a path in $Z_r$. Every point of $Z_r$ is joined by a segment in $Z_r$ to a point of $X$. Hence $Z_r$ is path-connected. Letting $r$ decrease to $b/2$ gives the infimum over path-connected spaces. Finally, if $X$ is compact, the set
\[
 T_r=\{(x,y,\lambda)\in X^2\times[0,1]:\lambda\|y-x\|\leqslant r\}
\]
is closed in the compact space $X^2\times[0,1]$. The continuous map $(x,y,\lambda)\mapsto(1-\lambda)x+\lambda y$ takes $T_r$ onto $Z_r$, so $Z_r$ is compact. A chain-connected compact space is connected: otherwise the two nonempty compact sets in a separation would be at positive distance.
\end{proof}

\subsection*{Which clouds contain a path-connected space?}
\begin{corollary}\label{cor:clouds}
For any metric space $X$, the following conditions are equivalent:
\begin{enumerate}
\setlength{\itemsep}{0pt}\setlength{\parskip}{0pt}
\item[\textup{(i)}] $[X]$ contains a path-connected space.
\item[\textup{(ii)}] $[X]$ contains a chain-connected space.
\item[\textup{(iii)}] $[\Ult(X)]=[\Delta_1]$, or equivalently $\diam\Ult(X)<\infty$.
\end{enumerate}
In \textup{(iii)}, $X$ may be replaced by any member of its cloud. If $b=\diam\Ult(X)<\infty$, then for every $\vp>0$ there is a path-connected $P\in[X]$ such that
\[
 b/2\leqslant\dgh(X,P)\leqslant b/2+\vp.
\]
\end{corollary}

\begin{proof}
A path-connected space is chain-connected. If $A\in[X]$ is chain-connected, then $\Ult(A)$ is a point. By stability, $\Ult(Y)$ lies at finite distance from this point for every $Y\in[X]$. Conversely, if $b<\infty$, take $P=Z_{b/2+\vp}$ in Theorem~\ref{thm:shorten}. The strict inequality $2(b/2+\vp)>b$ allows the chords of a finite chain to join any two points of $X$, and $\dgh(X,P)\leqslant b/2+\vp$. The lower bound follows from~\eqref{eq:b} with $b(P)=0$.
\end{proof}

\begin{corollary}\label{cor:ultraclouds}
The cloud $[X]$ contains an ultrametric space if and only if $E(X)<\infty$, or equivalently $[\Ult(X)]=[X]$. In this case $\Ult$ is a $1$-Lipschitz retraction of $[X]$ onto its ultrametric members. The cloud of bounded spaces $[\Delta_1]$ is the only cloud containing both an ultrametric space and a chain-connected space.
\end{corollary}

\begin{proof}
The finite-defect criterion $E(X)<\infty$ follows from~\eqref{eq:fit}. When it holds, the quotient map $X\to\Ult(X)$ has distortion $E(X)$, so $\Ult(X)\in[X]$. Conversely, if $[\Ult(X)]=[X]$, then $\Ult(X)$ itself is an ultrametric representative. The retraction assertion follows from stability and the fact that $\Ult(V)=V$ up to isometry for every ultrametric $V$. If the cloud also contains a chain-connected space, Corollary~\ref{cor:clouds} gives $[X]=[\Ult(X)]=[\Delta_1]$. The one-point space $\Delta_1$ has both properties and belongs to this cloud.
\end{proof}

\begin{example}
All three cases occur among clouds of unbounded spaces. Such a cloud may contain a path-connected representative, an ultrametric representative, or neither. The cloud of $\RR$ contains a path-connected space, and the cloud of any unbounded ultrametric space contains an ultrametric space. For $X=\{p^n:n\geqslant0\}\subset\RR$, $p>1$, neither type of representative occurs. Indeed, for $n\geqslant1$,
\[
 u_X(1,p^n)=(p-1)p^{n-1},\qquad
 d_X(1,p^n)-u_X(1,p^n)=p^{n-1}-1.
\]
Every chain joining $1$ to $p^n$ must cross the gap from $p^{n-1}$ to $p^n$, and the chain of successive powers has mesh equal to the length of that gap. Thus both $b(X)$ and $E(X)$ are infinite.
\end{example}

\section*{Acknowledgements}
The author thanks his scientific advisor A.~A.~Tuzhilin, A.~O.~Ivanov, and the participants in their seminar \emph{Theory of Extremal Networks} for their continued interest in this work. This research was supported by grant 24-8-2-15-1 of the Foundation for the Advancement of Theoretical Physics and Mathematics.

\end{document}